\documentclass[]{amsart}

\usepackage[T1]{fontenc}
\usepackage[cp1250]{inputenc}

\usepackage{graphicx,lpic}
\usepackage{enumitem, caption}
\usepackage[sc]{mathpazo}
\usepackage{amssymb, amsthm, amscd, amsmath}

\theoremstyle{plain}
\newtheorem{theorem}{Theorem}[section]
\newtheorem{proposition}[theorem]{Proposition}

\theoremstyle{definition}
\newtheorem{definition}[theorem]{Definition}
\newtheorem{example}[theorem]{Example}
\newtheorem{remark}[theorem]{Remark}
\newtheorem{question}[theorem]{Question}

\begin{document}

\title{On a banded link presentation of knotted surfaces}
\dedicatory{Dedicated to J\'ozef H. Przytycki on his sixtieth birthday}
\author{Micha\l\ Jab\l onowski}
\date{\today}
\address{Institute of Mathematics, University of Gda\'nsk\\
Wita Stwosza 57, 80-952 Gda\'nsk, Poland\\
E-mail: michal.jablonowski@gmail.com}

\thanks{Research of the author was partially supported by grant BW 538-5100-B854-15}

\subjclass[2010]{Primary 57Q45; Secondary 57M99, 57R65.}

\keywords{knotted surfaces, diagrams, moves, banded link, flat form}

\begin{abstract}

We will discuss a method for visual presentation of knotted surfaces in the four space, by examining a number and a position of its Morse's critical points. Using this method, we will investigate surface-knot with one critical point of index $1$. Then we show infinitely many mutually distinct surface-knots that has an embedding with two critical points of index $1$. Next we define a long flat form of a banded link for any surface-knot and show diagrammatically a long flat form of $n$-twist-spun $(2,t)$-torus knots.
\end{abstract}

\maketitle

\section{Introduction}

In this paper we will look closer to a visual method of presentation of knotted surfaces in $\mathbb{R}^4$, by examining a number and a position of its Morse's critical points. Presenting diagrams of surfaces as embeddings of graph-like structures in $\mathbb{R}^3$ may be a good method for people not accustomed to the forth-dimension, to ponder open questions about knotted surfaces, especially of popular families of knotted spheres $\mathbb{S}^2$ obtained by a twist-spinning of classical knots.

The paper is organized as follows. First we give basic definitions and theorems needed for our visual presentation of knotted surfaces in the four space, with examples of trivial and nontrivial banded links. Later we will present properties concerning the number of maxima, minima and saddle points of closed surfaces embedded in $\mathbb{R}^4$. 

Then we define a band number of a surface-knot $F$, denoted by $bn(F)$, as the minimum number of bands in a banded link among all of banded links for any surface-knot equivalent to $F$. Using this method, of classical unknotted links with bands attached to them in $\mathbb{R}^3$, we will show that if a banded link has one band (i.e. the corresponding surface-knot has one critical point of index 1) then it is unknotted. Then we show that there are infinitely many mutually inequivalent surface-knots with the band number equal to 2. Next we define a long flat form of a banded link for any surface-knot and show a long flat form of $n$-twist-spun $T(2,2k+1)$-torus knots diagrammatically.

\section{Basic definitions and theorems}

We will be assuming that all manifolds have the standard smooth structure and maps between them are smooth. An embedding (or its image) of a closed surface into $\mathbb{R}^4$ is called the \emph{knotted surface}, or sometimes \emph{surface-knot} if it is connected. Two knotted surfaces are \emph{equivalent} or have the same \emph{type}, if there exists an orientation preserving auto-homeomorphism of $\mathbb{R}^4$, taking one of those surfaces to the other.

\subsection{Hyperbolic splitting and marked graph diagrams}

To describe a knotted surface in $\mathbb{R}^4$, we will use transverse cross-sections $\mathbb{R}^3\times\{t\}$ for $t\in\mathbb{R}$, denoted by $\mathbb{R}^3_t$. This method introduced by R.H. Fox and J.W. Milnor was later broadly presented in \cite{Fox62}.

\begin{theorem}[Lomonaco \cite{Lom81}, Kawauchi, Shibuya, and Suzuki \cite{KSS82}]

For any knotted surface $F$, there exists a knotted surface $F'$ satisfying the following:
\begin{enumerate}[label={(\roman*)}]
\item $F'$ is equivalent to $F$ and has only finitely many Morse's critical points.
\item All maximal points of $F'$ lie in $\mathbb{R}^3_1$.
\item All minimal points of $F'$ lie in $\mathbb{R}^3_{-1}$.
\item All saddle points of $F'$ lie in $\mathbb{R}^3_0$.
\end{enumerate}
\end{theorem}

We call a representation $F'$ in the theorem a \emph{hyperbolic splitting} of $F$. The zero section $\mathbb{R}^3_0\cap F'$ of the hyperbolic splitting gives us a 4-valent graph (with possible loops without vertices). We assign to each vertex a \emph{marker} (here a thin rectangle) that informs us about one of the two possible types of saddle points (see Fig. \ref{MJ_103}) depending on shapes of $\mathbb{R}^3_{\epsilon}\cap F'$ or $\mathbb{R}^3_{-\epsilon}\cap F'$ for a small number $\epsilon>0$. The resulting graph is called a \emph{marked graph}. 

\begin{figure}[ht]
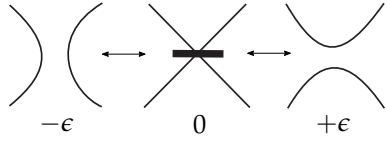

\begin{center}
\begin{lpic}[b(0.5cm)]{./figures/MJ_103(5cm)}
	\lbl[t]{15,-2;$-\epsilon$}
   \lbl[t]{60,-2;$0$}
   \lbl[t]{102,-2;$+\epsilon$}
	\end{lpic}
	\caption{Rules for a resolution of a marker.\label{MJ_103}}
\end{center}
\end{figure}

Making now a projection in general position of this graph to $\mathbb{R}^2$, and imposing crossing types like in the classical knot case, we receive a \emph{marked graph diagram} (terminology also applies to graphs of that kind which do not come from slicing closed surface).

\begin{theorem}[Kawauchi, Shibuya, and Suzuki \cite{KSS82}]

Let $F_i$ $(i=1,2)$ be a knotted surface in a hyperbolic splitting, and $D_i$ the marked graph diagram associated with the cross-section $F_i\cap\mathbb{R}^3_0$. If $D_1=D_2$, then $F_1$ is equivalent to $F_2$.
\end{theorem}

For a marked graph diagram $D$, we denote by $L_+(D)$ and $L_-(D)$ the diagrams obtained from $D$ by smoothing every vertex as shown in Fig. \ref{MJ_103} for $+\epsilon$ and $-\epsilon$, respectively. We have the following characterization of marked graph diagrams corresponding to knotted surface.

\begin{theorem}[Kawauchi, Shibuya, and Suzuki \cite{KSS82}]

Let $D$ be a marked graph diagram. There exists a knotted surface $F$ in a hyperbolic splitting such that $D$ is associated with the cross-section $F\cap\mathbb{R}^3_0$ if and only if $L_+(D)$ and $L_-(D)$ are diagrams of trivial links in $\mathbb{R}^3$.
\end{theorem}

\subsection{Calculus on links with bands}

A \emph{band} on a link $L$ is an embedded $I\times I$ intersecting the link $L$ precisely in the subset $b(\partial I\times I)$. We refer to the first factor of this $I\times I$ as the \emph{length} (with its boundary, the \emph{ends}). A {\it banded link} $BL$ in $\mathbb R^3$ is a pair $(L, B)$ consisting of a link $L$ in $\mathbb R^3$ and a finite set $B=\{b_0, \dots, b_m\}$ of pairwise disjoint $n$ bands spanning $L$.

\begin{figure}[ht]
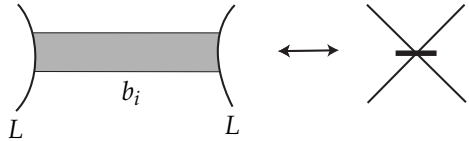

\begin{center}
\begin{lpic}[b(0.5cm)]{./figures/MJ_100(6cm)}
  \lbl[t]{30,8;$b_i$}
   \lbl[t]{0,-2;$L$}
   \lbl[t]{56,-1;$L$}
	\end{lpic}
	\caption{A band corresponding to a marked vertex.\label{MJ_100}}
\end{center}
\end{figure}

Let $BL =(L, B)$ be a banded link.  By an ambient isotopy of $\mathbb R^3$, we shorten the bands so that each band is contained in a small $2$-disk.  Replacing the neighborhood of each band to the neighborhood of a marked vertex as in Fig. \ref{MJ_100}, we obtain a marked graph, called a {\it marked graph associated with} $BL$. Conversely when a marked graph $G$ in $\mathbb R^3$ is given, by replacing each marked vertex with a band as in Fig. \ref{MJ_100}, we obtain a banded link $BL(G)$, called a {\it banded link associated with} $G$.  Moreover, when $G$ is described by a marked graph diagram $D$, then $BL(G)$ is also called a {\it banded link associated with} $D$, and denoted by $BL(D)$.

\begin{figure}[ht]
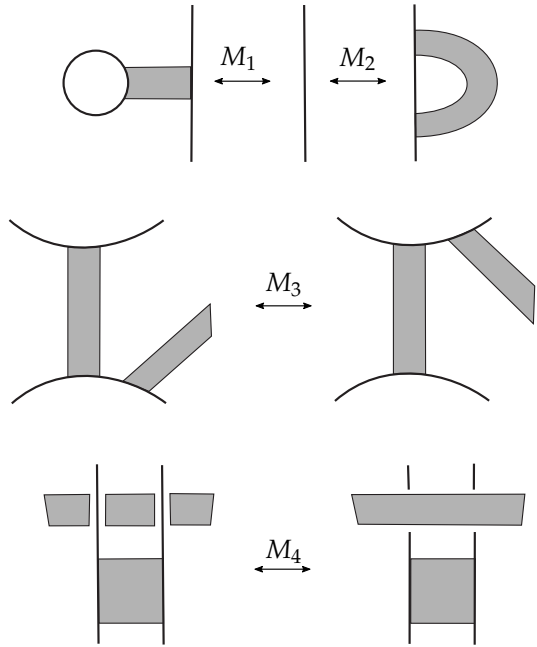

\begin{center}
\begin{lpic}[b(0.5cm)]{./figures/MJ_102(7cm)}
   \lbl[b]{56,140;$M_1$}
   \lbl[b]{85,140;$M_2$}
   \lbl[b]{67,85;$M_3$}
   \lbl[b]{67,20;$M_4$}
	\end{lpic}
	\caption{Allowed moves on a banded link.\label{MJ_102}}
\end{center}
\end{figure}

Let us also define the dual banded link $BL^*$ to $BL$ as a $BL(D^*)$ of a marked graph diagram $D^*$, which is obtained from a marked graph $D$ (associated with $BL$) by switching all of its markers to its second type (i.e. rotate marker by $90^{\circ}$). A \emph{hyperbolic transformation} of a link $L$ by band $b:I\times I\to \mathbb{R}^3$ is a link obtained by replacing $b(\partial I\times I)$ with $b(I\times \partial I)$. For a banded link $BL=(L,B)$ let as define $BL_-$ to be our starting link $L$ (but without attached bands), and $BL_+$ to be a link obtained from $L\cup B$ by a hyperbolic transformation (a surgery) on $L$ by every band from the set $B$.

The set of bands $B$ may be further partitioned into three sets, $B^U$ containing all fusion bands (i.e. decreasing number of components of $L$ in $BL_+$ after doing sequential hyperbolic transformations), $B^I$ containing all fission bands (i.e. increasing number of components) and $B^N$ containing all other bands.

\begin{theorem}[Swenton \cite{Swe01}, more details are in \cite{KeaKur08}]

Two knotted surfaces are equivalent if and only if, their banded links may be transformed one to another by an ambient isotopy in $\mathbb{R}^3$ and finite sequence of elementary local moves taken from the set $\{M_1, M_2, M_3, M_4\}$ presented in Fig. \ref{MJ_102}.
\end{theorem}

\begin{proposition}[Swenton \cite{Swe01}]

The connected sum of two connected surfaces is obtained by joining their banded links by a trivial band, as presented in Fig. \ref{MJ_101}. 
\end{proposition}

\begin{figure}[ht]
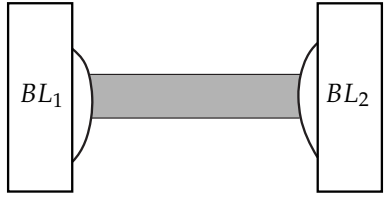

\begin{center}
\begin{lpic}[b(0.5cm)]{./figures/MJ_101(5cm)}
  \lbl[t]{8,25;$BL_1$}
  \lbl[t]{78,25;$BL_2$}
	\end{lpic}
	\caption{The connected sum of two surface-knots.\label{MJ_101}}
\end{center}
\end{figure}

\section{Properties of banded links}

An orientable embedded surface in $\mathbb{R}^4$ is \emph{unknotted} if it is equivalent to a surface embedded in $\mathbb{R}^3\times\{0\}\subset\mathbb{R}^4$. A banded link for an unknotted sphere is in Fig. \ref{MJ_106}(a), for an unknotted torus is in Fig. \ref{MJ_106}(d). An embedded projective plane $\mathbb{P}^2$ in $\mathbb{R}^4$ is \emph{unknotted} if it is equivalent to a surface, whose banded link looks like in Fig. \ref{MJ_106}(b) which is a \emph{positive} $\mathbb{P}^2_+$ or looks like in Fig. \ref{MJ_106}(c) which is a \emph{negative} $\mathbb{P}^2_-$. These are inequivalent surface and either of them is obtained by the reflection of the other. A banded link for unknotted Klein bottle is in Fig. \ref{MJ_106}(e).

\begin{figure}[ht]
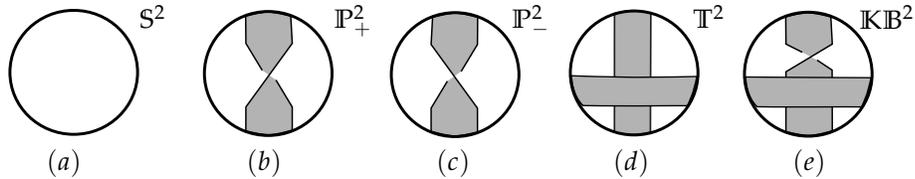

\begin{center}
\begin{lpic}[b(0.75cm)]{./figures/MJ_106(11.5cm)}
  \lbl[t]{26,24;$\mathbb{S}^2$}
  \lbl[t]{62,24;$\mathbb{P}^2_+$}
	\lbl[t]{94,24;$\mathbb{P}^2_-$}
  \lbl[t]{126,24;$\mathbb{T}^2$}
	\lbl[t]{158,24;$\mathbb{KB}^2$}
  \lbl[t]{10,-2;$(a)$}
  \lbl[t]{46,-2;$(b)$}
	\lbl[t]{80,-2;$(c)$}
  \lbl[t]{112,-2;$(d)$}
	\lbl[t]{144,-2;$(e)$}
	\end{lpic}
\caption{Examples of the unknotted surfaces.\label{MJ_106}}
\end{center}
\end{figure}

A non-orientable embedded surface in $\mathbb{R}^4$ is \emph{unknotted} if it is equivalent to some finite connected sum of unknotted projective planes. If a non-orientable surface $F$ is unknotted then $\pi_1\left(\mathbb{R}^4\backslash F\right)\approx\mathbb{Z}_2$.

The normal Euler number of a non-orientable surface-knot $F$ can only take on the following values: $2g, 2g-4,\ldots, 4-2g, -2g$, where $g$ is the number of $\mathbb{P}^2$ summands of $F$. This was conjectured by H. Whitney and proved by W.S. Massey, later S. Kamada gave another proof of this theorem in \cite{Kam89}.

\begin{proposition}\label{stw4}
The connected sum of a surface-knot with the  projective plane is represented by banded links as in Fig. \ref{MJ_107}(c) (the case of the negative projective plane is similar).
\end{proposition}

\begin{proof}
By definitions of the banded link for the positive projective and the definition of the banded link for the connected sum of surface-knots, we have the resulting banded link as in Fig. \ref{MJ_107}(a). Using two times the move $M_3$ (resulting in Fig. \ref{MJ_107}(b)), and the move $M_1$, we the obtain a banded link in Fig. \ref{MJ_107}(c).
\end{proof}

\begin{figure}[ht]
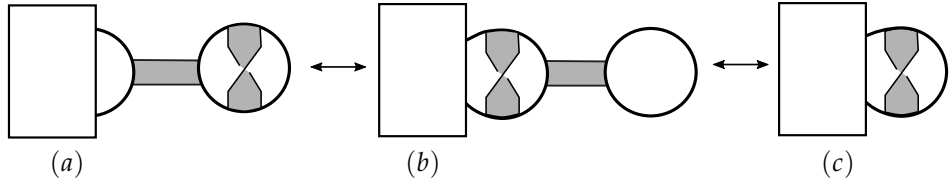

\begin{center}
\begin{lpic}[b(0.75cm)]{./figures/MJ_107(12.5cm)}
  \lbl[t]{10,-2;$(a)$}
	\lbl[t]{70,-2;$(b)$}
   \lbl[t]{140,-2;$(c)$}
	\end{lpic}
	\caption{The connected sum with projective plane.\label{MJ_107}}
\end{center}
\end{figure}

\begin{example}
Lets construct the marked graph diagram $D$ for the spin of the trefoil knot which corresponds to the closure of the surface singular braid word $a_2c_1^{-3}b_2c_1^{3}$, see \cite{Jab13} for the definition. Then we modify $BL(D)$ to obtain a banded link presented in Fig. \ref{MJ_104}. This surface-knot is known to be knotted, as an elementary computation of its fundamental group will show.
\end{example}

\begin{figure}[ht]
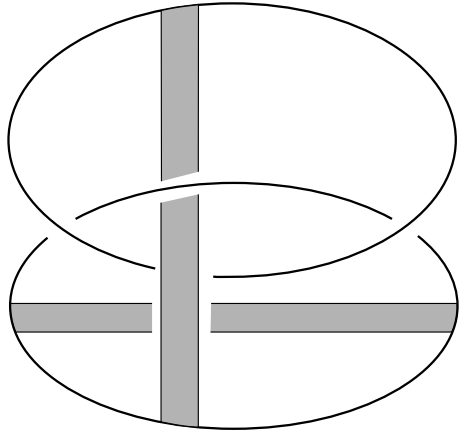

\begin{center}
\begin{lpic}[b(0.75cm)]{./figures/MJ_104(6cm)}
	\end{lpic}
	\caption{The spun of the trefoil knot.\label{MJ_104}}
\end{center}
\end{figure}

It can be easily verified that the Euler characteristic of a knotted surface $F$ associated to a banded link $BL=(L,B)$ satisfies the following formula. 

\begin{proposition}\label{stw1}

We have $\chi(F)=|BL_-|+|BL_+|-|B|$, where $|X|$ denotes the number of components in $X$.
\end{proposition}

We say that a banded link $BL\subset \mathbb{R}^3$ is in \emph{flat form} if the link $L$ is a split link embedded in some real plane $\mathbb{R}^2\subset \mathbb{R}^3$. It is clear that every knotted surface has a banded link presentation that is in flat form because $L$ is always the classical unlink.

We can verify easily the following.

\begin{proposition}\label{stw2}

If $BL$ is a banded link for a surface-knot (i.e. connected) then we have that $$|BL_+|+|BL_-|=|B^U|+|B^I|+2.$$
\end{proposition}

Combining Propositions \ref{stw1} and \ref{stw2}, then using definitions for fission and fusion bands, we have the following proposition which for the non-orientable case can be found in \cite{Kam89}.

\begin{proposition}\label{stw3}

If $F$ is a  surface-knot then $|B^N|=2-\chi(F)$.
\end{proposition}

\begin{definition}
The \emph{band number} of a surface-knot $F$, denoted by $bn(F)$, is the minimum number of bands in a banded link among all of banded links for $F$.
\end{definition}

It is easy to see that if $bn(F)=0$ then $F$ is the standard unknotted $2$-sphere $\mathbb{S}^2$. We will investigate this notion further.

\begin{theorem}\label{tw2}
If a surface-knot $F$ satisfies $bn(F)=1$ then $F$ is the unknotted projective plane $\mathbb{P}^2_+$ or $\mathbb{P}^2_-$.

\end{theorem}

\begin{proof}
Let $BL=(L,B)$ be a banded link for $F$, denote this one band by $b_1$. Because $F$ is connected, by Proposition \ref{stw2} we have that $|BL_-|\in\{1,2\}$. There are now three cases to consider.

\emph{Case 1:} the band $b_1$ is a fusion band. 
Then by Proposition \ref{stw2} we have that $|L|=|BL_-|=2$. The banded link $BL$ in a flat form is a split trivial link with attached band $b_1$. Because $BL_+$ is the unknot, then by \cite{Sch85} we have that the band $b_1$ is the trivial band, so we can delete it by the move $M_1$. That contradicts $bn(F)=1$.

\emph{Case 2:} the band $b_1$ is a fission band.
Then by Proposition \ref{stw2} we have that $|L|=1$. Take a dual banded link $BL^*$, the dual band $b_1^*$ to $b_1$ satisfy Case 1 above. We get that $F$ is a mirror reflection of the unknotted $S^2$, yielding $F$ to also be the unknotted $S^2$. That contradicts $bn(F)=1$.

\emph{Case 3:} the band $b_1$ is neither fusion nor fission band (i.e. $b_1\in B^N$ ).
Then by Proposition \ref{stw2} we have that $|L|=1$. The banded link $BL$ is the trivial knot with attached band $b_1$. Because $BL_+$ is the unknot, then by \cite{BleSch88} we have that $b_1$ is the trivial half-twisted band, so it is the unknotted projective plane $\mathbb{P}^2_+$ or $\mathbb{P}^2_-$.

\end{proof}

\begin{theorem}\label{tw3}
Let $F$ be the $n$-twist-spun ($n\not=\pm 1$) of nontrivial $2$-bridge classical knot, then $bn(F)=2$. 
\end{theorem}
\begin{proof}
Only for $n=\pm 1$ or trivial knot $K$ the $n$-twist-spun of knot $K$ is unknotted surface-knot (see \cite{Coc83}), so by Theorem \ref{tw2} we have that $bn(F)\geq 2$. To show that $bn(F)=2$ it is sufficient to construct a $BL$ for $F$ with two bands. We can do this directly by a marked graph diagram $D$ for plat closure of any $3$-strand braid $R$, which corresponds to the closure of the surface singular braid word $a_2Rb_2R^{-1}\Delta_3^{2n}$ (see \cite{Jab13}). The diagram $D$ has two markers corresponding to elements $a_2$ and $b_2$, so $BL(D)$ has two bands.
\end{proof}

\begin{theorem}\label{tw4}
Let $n=0,1,2,\ldots$  then there exists a surface-knot $F$ such that $bn(F)=n$. 
\end{theorem}
\begin{proof}
Cases for $n=0,1,2$ have already been discussed above. For $n>2$ take $F_n$ as a connected sum of $n$ copies of the unknotted projective plane (their signs may vary). From Proposition \ref{stw4}, the surface-knot $F_n$ has a banded link with $n$ bands, and it cannot have less because from Proposition \ref{stw3} we have that $bn(F_n)\geq |B^N|=2-(2-n)=n$.
\end{proof}

\begin{remark}\label{uwa1}
Similarly to the proof of Theorem \ref{tw4}, taking $T_n$ as an orientable unknotted surface-knots of genus $n$ for any $n=0,1,2,\ldots$, we can show that  $bn(T_n)=2n$ .
\end{remark}

Hence we find the following question interesting.

\begin{question}\label{pyt1}
Is the number $bn(F)$ for any orientable surface-knot $F$ even? 
\end{question}

\begin{theorem}\label{tw5}

For any surface-knot $F$ with a banded link $BL=(L,B)$ with $L=\{a_1, \ldots, a_k\}$, there exists a banded link $BL_l=(L_l, B_l)$ for $F$, in flat form, with $L_l=\{c_1, \ldots, c_n\}$ and $B_l=\{b_1, \ldots, b_m\}$ and disks $d_i$ spanning $c_i$, for $i=1, \ldots, n$ ; $m\geq n-1$ ; $k\geq n$ and $n=1, 2, \ldots$ satisfying the following:
\begin{enumerate}[label={(\roman*)}]

\item The band $b_i$ has its ends on $c_i$ and $c_{i+1}$ for $i=1,\dots, n-1$.
\item All bands $b_n,\ldots, b_m$ has both its ends on $c_n$.
\item If $n\geq 2$ then $int(d_1)\cap B_l\not=\emptyset$.

\end{enumerate}
\end{theorem}

We call a representation $BL_l$ in the above proposition a \emph{long flat form} of $F$, the schematic picture of this form is in Fig. \ref{MJ_105}. Notice that $B_l^U=\{b_1, \ldots, b_{n-1}\}$, and we can sometimes minimize number the link components considering long flat form for the dual banded link $BL^*$.

\begin{figure}[ht]
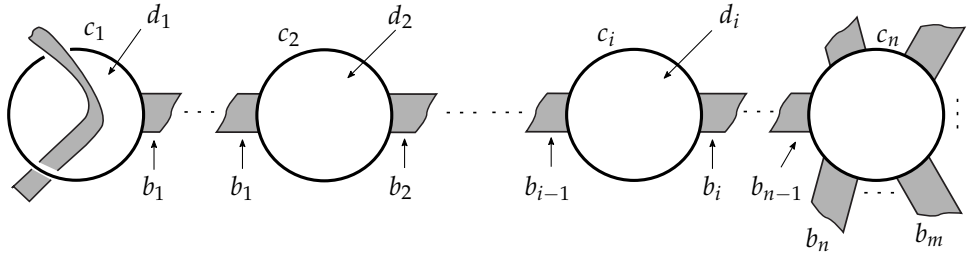

\begin{center}
\begin{lpic}[b(0.75cm),t(0.5cm)]{./figures/MJ_105(12.75cm)}
  \lbl[b]{15,33;$c_1$}
  \lbl[b]{48,32;$c_2$}
  \lbl[b]{102,32;$c_i$}
	\lbl[b]{150,32;$c_n$}
  \lbl[b]{26,35;$d_1$}
  \lbl[b]{67,35;$d_2$}
  \lbl[b]{123,35;$d_i$}
	  \lbl[t]{25,10;$b_1$}
  \lbl[t]{40,10;$b_1$}
  \lbl[t]{67,10;$b_2$}
	  \lbl[t]{92,10;$b_{i-1}$}
  \lbl[t]{120,10;$b_{i}$}
  \lbl[t]{131,10;$b_{n-1}$}
	\lbl[t]{138,1;$b_{n}$}
	\lbl[t]{157,2;$b_{m}$}
	\end{lpic}
\caption{A long flat form of a surface-knot.\label{MJ_105}}
\end{center}
\end{figure}

\begin{proof}

For $BL$ in flat form let us associate a multi-graph $M=(V, E)$ consisting of vertices $v_i$ corresponding to circles $a_i$, and edges between $v_p$ and $v_l$ corresponding to bands from $B$ having its both ends on $a_p$ and $a_l$. If the knotted surface associated to $BL$ is connected then this multi-graph must be connected. Let us fix the circle $a_k$ and take spanning tree $T$ of $M$ with root $v_k$. Orient all its edges towards the root.

Starting from leaves, slide by the move $M_3$ all bands corresponding to edges not in the tree $T$ from $a_r$ to $a_j$ if $v_r$ is a child of $v_j$. Continue the algorithm to obtain the circle $a_k$ having exactly $k-1$ bands attached only by one of its ends, and the other circles $a_1, \dots, a_{k-1}$ each having exactly one band attached and only by one of its ends. Label those bands $b_i$ if it is a band connected to $a_i$ for $i=1, 2, \ldots, k-1$.

We now can slide bands from $B$ to obtain an arc $\gamma\subset a_k$ that has attached bands $b_1, \dots, b_{k-1}$ in the monotonic order of its indices, and none of other bands attached to $\gamma$. If its not monotonic or other band (from $B\backslash \{b_1, \dots, b_{k-1}\}$) are attached, we can slide the bands by the move $M_3$ possibly around $a_i$ for $i<n$, over band $b_i$, to the other side of the arc $\gamma\backslash b_i$. Now slide the end of a band $b_1$ attached to $\gamma$ over the band $b_2$ to $a_2$, slide the end of a band $b_2$ attached to $\gamma$ over the band $b_3$ to $a_3$, etc. obtaining a banded link satisfying conditions $(i)$ and $(ii)$ by relabeling $a_i$ with $c_i$ and $k$ with $n$.

The condition $(iii)$ we can obtain as follows. If by contrary we have $int(d_1)\cap B_l=\emptyset$ then the banded link $BL$ in a long flat form is a split trivial link with attached band $b_1$, then similar to Case 1 in the proof of Theorem \ref{tw2}, we have that the band $b_1$ is the trivial band, so we can delete it by the move $M_1$. Then we proceed then with the new relabeled banded link (with indices decreased by one) in a long flat form as we obtain $d_1\cap B_l\not=\emptyset$, otherwise $n<2$ that contradicts the assumption. 
\end{proof}

\begin{figure}[ht]
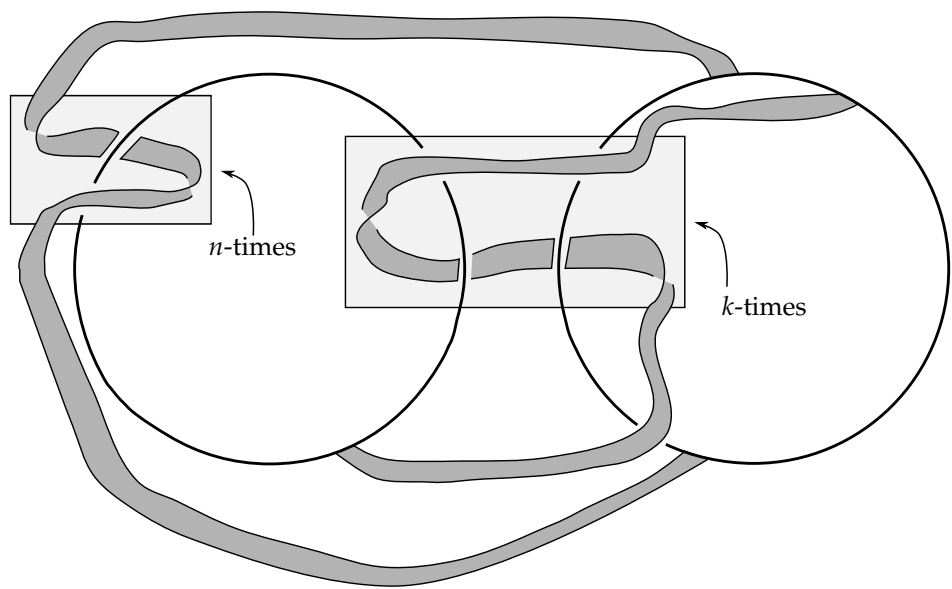

\begin{center}
\begin{lpic}[b(0.75cm)]{./figures/MJ_113(12.5cm)}
  \lbl[t]{140,54;$k$-times}
	\lbl[t]{45,65;$n$-times}
	\end{lpic}
	\caption{A long flat form of $n$-twist-spun of $T(2,2k+1)$ knot.\label{MJ_109}}
\end{center}
\end{figure}

\begin{example}
The marked graph diagram $D$ for the $n$-twist-spun of classical torus $T(2,2k+1)$ knot, for $n,k\in\mathbb{Z}$, corresponds to $3$-strand closure of the surface singular braid word $a_2c_1^{-2k-1}b_2c_1^{2k+1}\Delta^{2n}$ (see \cite{Jab13}). We modify a banded link $BL(D)$ only by an ambient isotopy in $\mathbb{R}^3$ to obtain a banded link in a long flat form presented in Fig. \ref{MJ_109}, where on the right side a band wraps around two arcs $k$ number of times and on the left side, the other band wraps around one arc $n$ number of times. This surface-knot is known to be unknotted $\mathbb{S}^2$ only for $n=\pm1$ or $k=0$ or $k=-1$.
\end{example}

\begin{theorem}\label{tw6}

Let $BL_l$ be a long banded link for a surface-knot $F$ with $|L|=n$ and assume that there exists half-twisted $b\in B^N$ such that $b\cap (d_1\cup d_2\cup \ldots\cup d_{n-1})=\emptyset$. Then $F$ is the connected sum of the unknotted projective plane and some surface-knot.

\end{theorem}

\begin{proof}

The band $b$ is an element of $B^N$ so its ends are attached to $c_n$. Because $b\cap (d_1\cup d_2\cup \ldots\cup d_{n-1})=\emptyset$, let us shrink $b$ maximally, keeping all circles from $L$ in the fixed position, but allowing bands from $B$ to smoothly isotope along the way. The core of the band $b$ divides now disc $d_n$ into two sides left $d_n^L$ and right $d_n^R$, make the move $M_1$ introducing a new trivial band $t$ attached to the arc $(c_n\backslash b)\cap d_n^R$ and pass all bands intersecting $int(d_n^R)$ by making moves of type $M_4$. Slide then all band ends attached to $(c_n\backslash b)\cap d_n^R$ including that of $t$, by moves of type $M_3$ along $b$, to the arc $(c_n\backslash b)\cap d_n^L$. The band $b$ was half-twisted so $int(d_n^R)\cap (B\backslash b)=\emptyset$ making the most right of long banded link looks like in Fig. \ref{MJ_107}(c), which by Proposition \ref{stw4} makes $F$ the connected sum of the unknotted projective plane and some surface-knot.

\end{proof}

\begin{remark}
In Fig. \ref{MJ_112} we present flat forms for the connected sum of unknotted projective plane $\mathbb{P}^2_+$ with spun trefoil (i.e. $\tau^0(T(2,3))\sharp \mathbb{P}^2_+$) and the connected sum of unknotted projective plane $\mathbb{P}^2_+$ with $2$-twist-spun trefoil (i.e. $\tau^2(T(2,3))\sharp \mathbb{P}^2_+$). These two surface-knots differ only in top of the picture. There is an open question whether they are equivalent (see \cite{BCCK14}).
\end{remark}

\begin{figure}[ht]
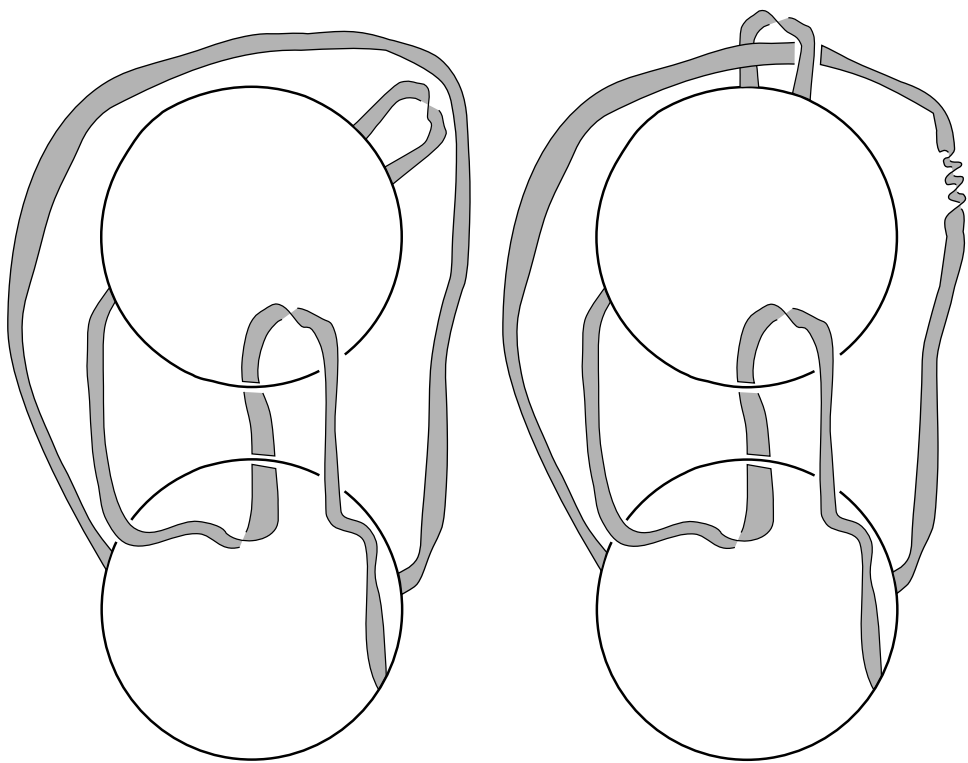

\begin{center}
\begin{lpic}[b(0.75cm)]{./figures/MJ_114(12.75cm)}
	\end{lpic}
	\caption{A flat form of $\tau^0(T(2,3))\sharp \mathbb{P}^2_+$ and of $\tau^2(T(2,3))\sharp \mathbb{P}^2_+$.\label{MJ_112}}
\end{center}
\end{figure}

\section*{Acknowledgments}
Research of the author was partially supported by grant BW 538-5100-B854-15.

\end{document}